\documentclass[12pt,final]{elsarticle}

\usepackage{amssymb}

\usepackage{amssymb}

\usepackage{rotating}

\usepackage{natbib}

\usepackage{multirow}

\usepackage{color}

\usepackage{soul} 

\usepackage{hyperref}

\newcommand{\bfg}[1]{\mbox{\boldmath $#1$\unboldmath}}

\newcommand{\fraca}[2]{\displaystyle\frac{#1}{#2}}

\newcommand{\mm}[3]{\renewcommand{\arraystretch}{0.8}\begin{array}[t]{c}\mbox{#1}
\\ #2\end{array}\begin{array}[t]{c}#3\end{array}
\renewcommand{\arraystretch}{1}}

\def \R {{\rm I\kern -2.2pt R\hskip 1pt}}

\newtheorem{algorithm1}{Algorithm}[section]

\newtheorem{theorem}{Theorem}[section]
\newtheorem{proposition}{Proposition}[section]
\newtheorem{proof}{Proof}[section]

\def\boxforqed{\rule{0.5em}{1.5ex}}
\def\qed{\ifmmode\squareforqed\else{\unskip\nobreak\hfil
        \penalty50\hskip1em\null\nobreak\hfil\boxforqed
         \parfillskip=0pt\finalhyphendemerits=0\endgraf}\fi}

\begin{document}

\begin{frontmatter}



\title{On the convergence of cutting-plane methods for robust optimization with ellipsoidal uncertainty sets}


\author[label1]{R.~M\'{\i}nguez\corref{cor1}}
\address[label1]{Dr. Engr., Independent Consultant, Ciudad Real, Spain}
\author[label2]{V.~Casero-Alonso}
\address[label2]{Department of Mathematics, Institute of Mathematics Applied to Science and Engineering, University of Castilla-La Mancha,
Ciudad Real, Spain}
\cortext[cor1]{Corresponding author: rominsol@gmail.com}

\begin{abstract}
Recent advances in cutting-plane strategies applied to robust optimization problems show that they are competitive with respect to problem reformulations and interior-point algorithms. However, although its application with polyhedral uncertainty sets guarantees convergence, finite termination when using ellipsoidal uncertainty sets is not theoretically guaranteed. This paper demonstrates that the cutting-plane algorithm set out for ellipsoidal uncertainty sets in its more general form also converges in a finite number of steps.
%
\end{abstract}

\end{frontmatter}

\section{Introduction}
%
%
The concept of robust optimization (RO) 
\cite{Soyster:73}--\cite{BertsimasS:04}
was developed to drop the classical assumption in mathematical programming that the input data is precisely known and equal to given nominal values.
Robust optimization techniques design solutions that are immune to data uncertainty
by solving equivalent deterministic problems. The main advantage of these techniques is that it is not required to know the probability density function of the uncertain data because the decision-maker searches for the optimal solution in view of the worst realization of uncertain data within the uncertainty set.

The most common uncertainty sets used in the technical literature are polyhedral and ellipsoidal. Robust problems with those uncertainty sets are usually solved using two different approaches: reformulation to a deterministic problem \citep{Ben-TalN:99}, also known as robust counterpart (RC), or iterative linear cutting-plane methods
\cite{FischettiM:12}--\cite{BertsimasDL:16}.
%
Recent advances in cutting-plane strategies show that, in terms of computational performance, there is no clear dominant method \cite{BertsimasDL:16}, which makes this type of approach competitive with respect to the reformulation method and interior-point algorithms, and it constitutes, indeed, a plausible alternative for practitioners and researchers. However, although its application with polyhedral uncertainty sets guarantees convergence, finite termination when using ellipsoidal uncertainty sets is not theoretically guaranteed, which might discourage its implementation in practice.

Kelley (1960) \cite{Kelley:60} was pioneer in developing a cutting-plane method for solving convex programs. He developed an algorithm that involves solving an infinite sequence of linear programs, which in practice is truncated after a finite number of steps to obtain an approximate solution within an specified tolerance. In addition, \cite{Kelley:60} sets out the fundamental convergence theorem which ensures convergence in a finite number of steps for continuous variables. Westerlund and Pettersson (1995) \cite{WesterlundP:95} extended the Kelley's cutting-plane method for solving convex mixed-integer nonlinear programming problems with a moderate degree of nonlinearity. Based on Kelley's convergence theorem, \cite{WesterlundP:95} demonstrates that the sequence of solutions reached by the cutting-plane method converges to the global optimal solution in a finite number of steps, even with the consideration of continuous and binary variables.

The aim of this paper is to demonstrate that the cutting-plane method for getting robust solutions for uncertain linear programs with ellipsoidal uncertainty sets is equivalent to the extended Kelley's cutting plane method proposed by \cite{WesterlundP:95}, which theoretically guarantees its convergence in a finite number of steps.

The rest of the paper is organized as follows. Section~\ref{s1} introduces the cutting-plane method for mixed-integer  robust optimization. In Section~\ref{s2} a brief description of the problem and cutting-plane scheme considered by \cite{WesterlundP:95} is given, in addition, the equivalency among formulations and cutting-plane strategies is demonstrated. Finally, in Section~\ref{Conclu} relevant conclusions are drawn.

\section{Cutting-plane method for mixed-integer robust optimization}\label{s1}
Consider the following general mixed-integer RO problem:
\begin{eqnarray}
\mm{Minimize}{{\bfg x}}{{\bfg c}^T{\bfg x}},&&\label{eq.fobj}\\
 \mbox{subject to \hspace*{0.5cm}}\tilde{\bfg a}_i^T {\bfg x} &\le & {\bf 0};\forall  \tilde{\bfg a}_i\in U_i, i=1,\ldots,m,\label{eq.incons}\\
{\bfg l}\le {\bfg x} & \le &{\bfg u},\label{eq.bounds}\\
x_j &\in &{\cal Z};\;\forall j\le k;\;1\le k\le n,\label{eq.integrality}
\end{eqnarray}
where ${\bfg x}\in {\cal R}^n$ is the decision variable vector of $n$ variables, the first $k$ of which are integral. ${\bfg c}(n \times 1)$ are data coefficients, ${\bfg l}(n \times 1)$ and ${\bfg u}(n \times 1)$ are lower and upper decision variable bounds, and $\tilde{\bfg a}_i(n \times 1)$ are uncertain coefficients, where the tilde indicates its uncertain character. $U_i$ is the ellipsoidal uncertainty set defining the possible values of uncertain parameters $\tilde{\bfg a}_i$, being ${\bfg a}_i(n \times 1)$ the vector of nominal values around which uncertain parameters may vary. For those cases where vector ${\bfg c}(n \times 1)$ is uncertain, or even the right hand side of equation~(\ref{eq.incons}) is uncertain and equal to ${\bfg b}(m \times 1)$, it is possible to rewrite the original problem as (\ref{eq.fobj})-(\ref{eq.integrality}) (see \cite{Ben-TalEN:09}).

The RC of problem (\ref{eq.fobj})-(\ref{eq.integrality}) replaces constraint set (\ref{eq.incons}) by:
\begin{equation}\label{eq.inconsRCellip}
   \left(\mm{Maximum}{\tilde{\bfg a}_{i}\in U_i}{\tilde{\bfg a}_i^T {\bfg x}}\right) \le  0,\; i=1,\ldots,m.
\end{equation}

Thus, the RC of problem (\ref{eq.fobj})-(\ref{eq.integrality}) using ellipsoidal uncertainty sets becomes:
\begin{eqnarray}
&&\mm{Minimize}{{\bfg x}}{{\bfg c}^T{\bfg x}},\label{eq.fobjF}\\
&&\mbox{\hspace*{0.2cm}subject to      }  {\bfg a}_i^T{\bfg x}+\beta_i\sqrt{{\bfg x}^T{\bfg \Sigma}_i{\bfg x}}  \le   0; i=1,\ldots,m,\label{eq.inconsF}\\
&& \hspace*{2.2cm} {\bfg l}\le {\bfg x}  \le  {\bfg u},\label{eq.boundsF}\\
&&  \hspace*{2.2cm} x_j  \in  {\cal Z};\; \forall j\le k;\;1\le k\le n,\label{eq.integralityF}
\end{eqnarray}
where positive parameters $\beta_i$ and the positive-definite variance-covariance matrices ${\bfg \Sigma}_i$, associated with uncertain parameters $\tilde{\bfg a}_i$, control the size and protection level of the ellipsoidal set for each constraint.

The cutting-plane method is a decomposition technique, such as Benders' decomposition \citep{Floudas:95,ConejoCMG:06}, where instead of dealing with the full mixed-integer second-order cone programming problem (\ref{eq.fobjF})-(\ref{eq.integralityF}) at once, it is decomposed into a linear master problem and different quadratically constrained subproblems with analytical solution related to constraint (\ref{eq.inconsRCellip}). The optimal solution is achieved by solving these master and subproblems in an iterative fashion as described in the following algorithm.

\begin{algorithm1}[Basic cutting-plane algorithm]\label{alg}
$\;$
\begin{enumerate}
\item[Step 0] \textbf{Initialization:}
    Set the tolerance of the process $\varepsilon$, the initial value of the iteration counter $l=0$ and the initial values of the uncertain parameters to their nominal values $\hat{\bfg a}_i^{(0)}= {\bfg a}_i;\forall i$.
\item[Step 1] \textbf{Solving the master problem:} Update the iteration counter $l \longrightarrow l+1$ and calculate the optimal solution ${\bfg x}_l$ of the following master problem:
\begin{eqnarray}
    \mm{minimize}{{\bfg x}}{{\bfg c}^T{\bfg x}},\label{eq.fobjF2.Step1}\\
 \mbox{subject to \hspace*{0.1cm} }\hat {\bfg a}_i^{(\nu)^T}{\bfg x} & \le & 0;\;i=1,\ldots,m;\;\nu=0,1,\ldots,l-1, \label{eq.inconsF2.Step1}\\
  {\bfg l}\le {\bfg x} & \le & {\bfg u},\label{eq.boundsF2.Step1}\\
  x_j & \in & {\cal Z};\; \forall j\le k;\;1\le k\le n.\label{eq.integralityF2.Step1}
\end{eqnarray}
Continue in {\bf Step 2}.
\item[Step 2] \textbf{Stopping rule:} Check if the current solution $\bfg{x}_l$ satisfies the original conic restrictions (\ref{eq.inconsF}). If it does, the optimal solution has been found. If it does not and $\hat {\bfg a}_i^{(\nu)^T}{\bfg x}_l<\varepsilon;\forall i$, stop the process with ${\bfg x}_l$ as optimal solution, otherwise continue to {\bf Step 3}.
\item[Step 3] \textbf{Solving subproblems:} Plug solution ${\bfg x}_l$ obtained in {\bf Step 1} into the following expression:
 \begin{equation}\label{eq.ahati}
    \hat {\bfg a}^{(l)}_i={\bfg a}_i+\beta_i\fraca{{\bfg \Sigma}_i{\bfg x}_l}{\sqrt{{\bfg x}_l^T{\bfg \Sigma}_i{\bfg x}_l}};\;i=1,\ldots,m,
\end{equation}
which corresponds to the analytical solution of subproblems (\ref{eq.inconsRCellip}) \cite[Lemma 1]{Ackooij:17}, and continue in {\bf Step 1}.
\end{enumerate}
\end{algorithm1}

To our knowledge, convergence of this algorithm in a finite number of steps is not theoretically guaranteed, which might have discouraged its practical implementation to date. Nevertheless, there are many variants of this algorithm proposed in the technical literature \cite{BertsimasDL:16} that makes this method computationally competitive and practical in many cases.

\section{On the convergence of cutting plane algorithm for RO with ellipsoidal uncertainty sets}\label{s2}
Westerlund and Pettersson (1995) \cite{WesterlundP:95} proposed an iterative algorithm to solve the following problem:
\begin{eqnarray}
&&\mm{Minimize}{{\bfg x}}{{\bfg c}^T{\bfg x}},\label{eq.fobjFWP95}\\
&&\mbox{\hspace*{0.2cm}subject to      }  {\bfg g}({\bfg x})  \le   {\bf 0},\label{eq.inconsFFWP95}\\
&&  \hspace*{2.2cm} x_j  \in  {\cal Z};\; \forall j\le k;\;1\le k\le n,\label{eq.integralityFFWP95}
\end{eqnarray}
where the vector of variables ${\bfg x}=({\bfg x}^{\rm d}; {\bfg x}^{\rm c})$ includes discrete (${\bfg x}^{\rm d}\in{\cal Z}^k$) and continuous (${\bfg x}^{\rm c}\in{\cal R}^{n-k}$) variables simultaneously, ${\bfg g}({\bfg x})$ are $m$ continuous convex functions, all defined on a set ${\cal L}=X^{\rm d}\cup X^{\rm c}$, where $X^{\rm d}$ is a finite discrete set defined by
$X^{\rm d}=\left\{{\bfg B}{\bfg x}^{\rm d}\le {\bfg b}\right\}$, and $X^{\rm c}$ is a $(n-k)$-dimensional compact polyhedral convex set defined by
$X^{\rm c}=\left\{{\bfg A}{\bfg x}^{\rm c}\le {\bfg e}\right\}$.

The algorithm set out by Westerlund and Pettersson (1995) \cite{WesterlundP:95} establishes a sequence of points ${\bfg x}_\nu; \nu=0,1,\ldots,l-1$, replacing original constraints (\ref{eq.inconsFFWP95}) by:
\begin{equation}\label{cuttingKelley}
 g_i({\bfg x}_\nu)+
 \left(\fraca{\partial g_i}{\partial {\bfg x}}\right)^T_{{\bfg x}_\nu}({\bfg x}-{\bfg x}_\nu)\le 0;\;i=1,\ldots,m;\;\nu=0,1,\ldots,l-1.
\end{equation}

Since ${\bfg g}({\bfg x})$ are $m$ continuous convex functions, the following condition holds:
\begin{equation}\label{convexityKelley}
 g_i({\bfg x}_\nu)+
 \left(\fraca{\partial g_i}{\partial {\bfg x}}\right)^T_{{\bfg x}_\nu}({\bfg x}-{\bfg x}_\nu) \le g_i({\bfg x});\;i=1,\ldots,m;\;\nu=0,1,\ldots,l-1.
\end{equation}

Convexity, and therefore condition (\ref{convexityKelley}), is a requirement for the convergence of the method. Based on problem definition (\ref{eq.fobjFWP95})-(\ref{eq.integralityFFWP95}) the following proposition can be set out:
\begin{proposition}\label{prop1}
The RC (\ref{eq.fobjF})-(\ref{eq.integralityF}) responds to problem description (\ref{eq.fobjFWP95})-(\ref{eq.integralityFFWP95}) given by \cite{WesterlundP:95}.
\end{proposition}
\begin{proof}
Decision variables ${\bfg x}=({\bfg x}^{\rm d}; {\bfg x}^{\rm c}) (n\times 1)$ in (\ref{eq.fobjF})-(\ref{eq.integralityF}) are defined on a set ${\cal L}=X^{\rm d}\cup X^{\rm c}$ through lower and upper decision variable bounds (\ref{eq.bounds}). In addition, constraints in (\ref{eq.inconsF}) are convex. For any given ${\bfg x}_k\in {\cal L}$, condition (\ref{convexityKelley}) holds:
\begin{equation}\label{proof1}
  \begin{array}{l}
     g_i({\bfg x}_k)+
 \left(\fraca{\partial g_i}{\partial {\bfg x}}\right)^T_{{\bfg x}_k}({\bfg x}-{\bfg x}_k)=\\
  {\bfg a}_i^T{\bfg x}_k+\beta_i\sqrt{{\bfg x}_k^T{\bfg \Sigma}_i{\bfg x}_k} +\left({\bfg a}_i+\beta_i\fraca{{\bfg \Sigma}_i{\bfg x}_k}{\sqrt{{\bfg x}_k^T{\bfg \Sigma}_i{\bfg x}_k}}\right)^T({\bfg x}-{\bfg x}_k)=
  \\
{\bfg a}^T_i{\bfg x}+\beta_i\fraca{{\bfg x}_k^T{\bfg \Sigma}_i{\bfg x}}{\sqrt{{\bfg x}_k^T{\bfg \Sigma}_i{\bfg x}_k}}\le   {\bfg a}_i^T{\bfg x}+\beta_i\sqrt{{\bfg x}^T{\bfg \Sigma}_i{\bfg x}} = g_i({\bfg x})
   \end{array}
\end{equation}
 $\forall i=1,\ldots,m$ because ${\bfg x}_k^T{\bfg \Sigma}_i{\bfg x}\le
 \sqrt{{\bfg x}_k^T{\bfg \Sigma}_i{\bfg x}_k}
 \sqrt{{\bfg x}^T{\bfg \Sigma}_i{\bfg x}}
 $ due to Cauchy-Schwarz inequality. \qed

\end{proof}

Finally, the equivalency among cutting-planes (\ref{eq.inconsF2.Step1}) and (\ref{cuttingKelley}) is set out:
\begin{proposition}\label{prop2}
Cutting-planes (\ref{eq.inconsF2.Step1}) associated with robust problem solution are equivalent to those established by \cite{WesterlundP:95}, i.e. (\ref{cuttingKelley}).
\end{proposition}
\begin{proof}
For any given $i=1,\ldots,m$ and $\nu=0,\ldots,l-1$:
\begin{equation}\label{proof2}
     g_i({\bfg x}_\nu)+
 \left(\fraca{\partial g_i}{\partial {\bfg x}}\right)^T_{{\bfg x}_\nu}({\bfg x}-{\bfg x}_\nu)={\bfg a}^T_i{\bfg x}+\beta_i\fraca{{\bfg x}_\nu^T{\bfg \Sigma}_i{\bfg x}}{\sqrt{{\bfg x}_\nu^T{\bfg \Sigma}_i{\bfg x}_\nu}}
\end{equation}
which considering (\ref{eq.ahati}) are equal to $\hat {\bfg a}_i^{(\nu)^T}{\bfg x}$. \qed

\end{proof}

Propositions~\ref{prop1} and \ref{prop2} results in the following theorem:
\begin{theorem}
For any given RO problem using ellipsoidal uncertainty sets, the sequence of solution points ${\bfg x}_\nu$ generated by the cutting-plane algorithm~\ref{alg} converges to a solution of RC problem (\ref{eq.fobjF})-(\ref{eq.integralityF}) in a finite number of steps within tolerance $\epsilon$.
\end{theorem}
\begin{proof}
Propositions~\ref{prop1} and \ref{prop2} prove that both the problem definition of RC (\ref{eq.fobjF})-(\ref{eq.integralityF}) and the cutting-plane algorithm \ref{alg} are equivalent to those set out by \cite{WesterlundP:95}, for which a theoretical proof of convergence in a finite number of steps is given. Therefore, cutting-plane algorithm~\ref{alg} convergence is also theoretically  guaranteed in a finite number of steps. \qed
\end{proof}

\section{Conclusions}\label{Conclu}
Based on convergence properties for an extended Kelley's cutting plane method, this paper demonstrates that finite termination of the cutting-plane method applied to ellipsoidal uncertainty sets is theoretically guaranteed. This fact might encourage the implementation of this method in practice.

\section*{Acknowledgments}
%
Dr. M\'{\i}nguez is fully supported by the Public State Employment Service of the Ministry of Labour, Migration and Social Security of Spain. Dr. Casero-Alonso is sponsored by Ministerio de Econom\'{\i}a y Competitividad under grant contract FEDER MTM2016-80539-C2-1-R and Consejer\'{\i}a de Educaci\'on, Cultura y Deportes of Junta de Comunidades de Castilla-La Mancha under grant contract FEDER SBPLY/17/180501/000380.

\bibliographystyle{elsarticle-num}

\end{document}